\documentclass[12pt]{amsart} 
\usepackage[top=1in, bottom=1in, left=1in, right=1in]{geometry}

\makeatletter
\@namedef{subjclassname@2020}{%
  \textup{2020} Mathematics Subject Classification}
\makeatother

\usepackage{times}
\usepackage[table,dvipsnames]{xcolor}

\PassOptionsToPackage{hyphens}{url}\usepackage[colorlinks=true, linkcolor=NavyBlue, citecolor=teal, urlcolor=gray]{hyperref}   
\numberwithin{equation}{section}

\usepackage{scrextend}
\usepackage{amssymb,mathrsfs,amsmath}
\usepackage{mathtools}
\usepackage{bbm}
\usepackage{setspace}
\usepackage{enumitem}
\usepackage{hyperref}

\theoremstyle{plain}
\newtheorem{thm}{Theorem}
\newtheorem{prop}{Proposition}[section]
\newtheorem{lem}[prop]{Lemma}

\newtheorem{corollary}[thm]{Corollary}

\theoremstyle{definition}

\theoremstyle{remark}
\newtheorem*{rem*}{Remark}
\newtheorem*{rems*}{Remarks}
\newtheorem{rem}[prop]{Remark}

\newcommand{\Z}{\mathbb{Z}}

\newcommand{\CP}{\mathcal{P}}

\newcommand{\bv}\boldsymbol{}

\newcommand{\eps}\varepsilon

\renewcommand{\leq}{\leqslant}
\renewcommand{\le}{\leqslant}
\renewcommand{\geq}{\geqslant}
\renewcommand{\ge}{\geqslant}

\renewcommand{\mod}[1]{\ (\mathrm{mod}\,#1)}

\title{On extreme values of $r_3(n)$ in arithmetic progressions}

\author{Michael Filaseta}
\address{Mathematics Department \\
University of South Carolina \\
Columbia, SC 29208}
\curraddr{}
\email{filaseta@math.sc.edu}

\author{Jonah Klein}
\address{Mathematics Department \\
University of South Carolina \\
Columbia, SC 29208}
\curraddr{}
\email{jonah.klein@sc.edu}

\author{Cihan Sabuncu}
\address{D\'epartement de math\'ematiques et de statistique\\
Universit\'e de Montr\'eal\\
CP 6128 succ. Centre-Ville\\
Montr\'eal, QC H3C 3J7\\
Canada}
\curraddr{}
\email{cihan.sabuncu@umontreal.ca}

\dedicatory{Dedicated to George Andrews and Bruce Berndt on their 85th birthdays}


\begin{document}

\begin{abstract}
For a given integer \( m \) and any residue \( a \mod m \) that can be written as a sum of 3 squares modulo $m$, we show the existence of infinitely many integers \( n \equiv a \mod m \) such that the number of representations of $n$ as a sum of three squares, \( r_3(n) \), satisfies $r_3(n) \gg_m \sqrt{n} \log \log n $. Consequently, we establish that there are infinitely many integers \( n \equiv a \mod m \) for which the Hurwitz class number \( H(n) \) also satisfies \( H(n) \gg_m \sqrt{n} \log \log n \). 

\end{abstract}
\maketitle
\section{Introduction}
Writing positive integers as a sum of a fixed number of squares has a long and interesting history (cf.~\cite{grosswald}).  Of particular importance is the groundbreaking work of Ramanujan to this field. For an in-depth discussion of Ramanujan's contributions, see the recent paper by Berndt and Moree~\cite{berndtmoree}.
Our main focus in this paper will be on showing that, in any given arithmetic progression $a \mod{m}$ containing numbers which are sums of $3$ squares, there are $n$ for which the number of representations of $n$ as a sum of $3$ squares exceeds what one obtains on average by a factor of $\gg_{m} \log \log n$.  Another focus is to show that a certain approach of Chowla \cite{Chowla}, in the analogous problem where no arithmetic progression is specified, making use of counts on the number of representations of positive integers as a sum of $4$ squares, can be modified to give the corresponding result for arithmetic progressions.

For $n$ a positive integer, let
\[ r_3(n) := \big{|}\{ (a,b,c)\in \Z^3 : a^2+b^2+c^2=n \}\big{|}. \]
Legendre's Three-Square Theorem \cite{Legendre} asserts that $r_3(n)>0$ if and only if $n$ is not of the form $4^u(8v+7)$ for some non-negative integers $u$ and $v$. 

The problem of estimating the average number of representations of an integer as a sum of $3$ squares corresponds to counting lattice points in a sphere, from which we can see that
\[\sum_{0 \leq n \leq x} r_3(n) \sim \frac{4\pi}{3} x^{3/2}.\]
With this formula, we can deduce that there are infinitely many integers $n$ for which $r_3(n) \gg \sqrt{n}$. We refer to this bound as the \textit{trivial} bound. 

Chowla \cite{Chowla}, in the goal of finding large values of $L(1,\chi_{-4n})$, showed that there are infinitely many square-free integers $n$ congruent to $1 \mod 4$ such that
\[ r_3(n) \gg \sqrt{n} \log\log n.\]
The connection between $r_3(n)$ and $L(1,\chi_{-4n})$ comes from a classical result due to Gauss, who proved that\footnote{See \cite{Mortenson} for references and a recent proof of this result.}
\begin{equation}\label{Hurwitz_class}
r_3(n) = \begin{cases} 
      12H(n) & \text{if } n\equiv 1,2 \mod 4 \\
       24H(n) & \text{if } n\equiv 3 \mod 8\\
       r_3(n/4) & \text{if } n\equiv 0 \mod 4 \\
        0 & \text{if } n\equiv 7 \mod 8,
   \end{cases}
   \end{equation}
where $H(n)$ is the Hurwitz class number, which is defined for $-n=Df^2$ where $D$ is a negative fundamental discriminant as
\[ H(n)= \frac{\sqrt{|D|} \cdot L(1,\chi_D) }{\pi} \sum_{d \mid f} \mu(d) \chi_D(d) \sigma\bigg( \frac{f}{d}\bigg), \]
where $\chi_D = (\frac{D}{\cdot})$ is the Kronecker symbol, $\mu(n)$ is the M\"obius function, and $\sigma(n)$ is the sum of divisors function.  See also the work of Granville-Soundararajan \cite{GranSound}, where they study the distribution of values of $L(1,\chi_d)$ as $\chi_d$ ranges over primitive real characters and examine how often it attains the maximal value.

The average behavior of $r_3(n)$ is well understood when we restrict $n$ to belong to some arithmetic progression $a \mod m$. Similarly to the general case, we can show that
\[
\sum_{0 \leq n \leq x,n \equiv a \mod m} r_3(n) \sim \frac{4c_{a,m} \pi} {3m^3}x^{3/2},
\]
where $c_{a,m}$ is the number of solutions to $x_1^2+x_2^2+x_3^2 \equiv a \mod m$ in $\Z/m\Z$. It follows that whenever $c_{a,m} \neq 0$, there are infinitely many integers $n \equiv a \mod m$ such that $r_3(n) \gg c_{a,m}\sqrt{n}/m^2$.  If $c_{a,m} = 0$, then $r_{3}(n) = 0$ for all $n \equiv a \mod m$.  
Thus, the trivial bound also holds when $n$ is restricted to an arithmetic progression representing numbers which are sums of $3$ squares.

Relationship \eqref{Hurwitz_class}, together with the definition of $H(n)$, tells us that $r_3(n)$ is intimately related to $L(1,\chi_D)$. As such, finding extreme values of $L(1,\chi_D)$, where $D$ is a negative fundamental discriminant, would allow one to find similar extreme values for $r_3(n)$.  Along these lines, the work of Brink, Moree, and Osburn \cite[Theorem 3.1]{Brink}, following up on a result of Joshi \cite{Joshi}, implies that certain cases of integers $m > 0$ and $a$ with $(a,m)=1$ are such that there are infinitely many integers $n$ for which $n \equiv a \mod m$ and $r_3(n) \gg_{m} \sqrt{n}\log\log n$.   

In this paper, we obtain $r_3(n) \gg_{m} \sqrt{n}\log\log n$ only requiring that $a \mod{m}$ contains integers which are the sums of $3$ squares. The approach we use is different, and we demonstrate in fact that one may modify an idea of Chowla \cite{Chowla} to establish our main theorem.  A variation of the results of Brink, Moree, and Osburn \cite{Brink}, and Joshi \cite{Joshi}, would provide an alternative approach to obtaining a result similar to the one obtained here. 

\begin{thm}\label{mainresult}
    Let $m$ be a positive integer, and let $a$ be any integer that can be written as a sum of $3$ squares modulo $m$. Then there exist infinitely many integers $n\equiv a \mod m$ such that as $n\to \infty$, we have
    \[
    r_3(n) \gg_m \sqrt{n}\log\log n.
    \]
\end{thm}

By the relationship \eqref{Hurwitz_class}, we also obtain lower bounds on the Hurwitz class number for infinitely many integers $n$ in arithmetic progressions. 

\begin{corollary}\label{maincor}
Let $m$ be a positive integer, and let $a$ be any integer not divisible by $4$ that can be written as a sum of 3 squares modulo $m$. Then there exist infinitely many integers $n\equiv a \mod m$ such that as $n\to \infty$, we have
\[
H(n) \gg_m \sqrt{n}\log\log n.
\]
\end{corollary}
\begin{rem}
    If we could force $n$ to be square-free, we would also get lower bounds on $L(1,\chi_{-n})$ where $n\equiv a \mod m$. 
\end{rem}

The paper is organized as follows. In Section \ref{overview}, we go over the general methodology that we will use to obtain Theorem \ref{mainresult}. In Section \ref{construction}, we show how to find solutions to $2$ needed simultaneous congruences in $4$ variables modulo a prime power to complete the proof of Theorem \ref{mainresult}. In Section \ref{corollaryproofs}, we show how to obtain Corollary \ref{maincor} from Theorem \ref{mainresult}. 

\section{Overview of the method}\label{overview}

Fix a positive integer $m$, and let
\[m=\prod_{i=1}^{\ell} p_i^{\nu_i}\]
be its unique factorization into primes. 
Fix some $a \in \{ 0,1, \ldots, m-1 \}$ such that $a$ is a sum of three squares modulo $m$.  Let $z$ be large, and set
\[
N_{1} = \prod_{p \in \CP} p,
\]
where $\CP$ is the set of primes $\le z$ excluding the primes $p_i$ for $1 \le i \le \ell$.  
Let $d \in \{ 0,1, \ldots, m-1 \}$ be the inverse of $N_{1}$ modulo $m$.
Let $N := a d N_{1}$.  Then $N \equiv a \pmod{m}$.
We take $z$ to be sufficiently large compared to $m$, noting, in particular, the largest prime $p\leq z$ satisfies $p \sim \log N_1 \sim \log N$ as $z$ (or $N$) tends to $\infty$.
Define
\[r_4(n)=\big{|}\{(k_1,k_2,k_3,k_4) \in \Z^4: k_1^2+k_2^2+k_3^2+k_4^2=n\}\big{|}.\]
By Jacobi's Four-Square Theorem, we have that
\[
r_4(n) = 8 \sigma_{*}(n),
\qquad \text{where }
\sigma_{*}(n) := \sum_{\substack{d \mid n \\ 4 \nmid d}} d.
\]
Thus, 
\begin{align*}
r_4(N) \gg_{m} N\prod_{p \in \CP}(1+1/p) \gg_m N\log\log N.
\end{align*}
Let 
\[E=\{(x_1,x_2,x_3,x_4) \in \{ 1, 2, \ldots, m\}^4: x_1^2+x_2^2+x_3^2+x_4^2 \equiv a \mod m\},\]
and for $(x_1,x_2,x_3,x_4) \in E$, let 
\[r_4(N,x_1,x_2,x_3,x_4)=\{(k_1,k_2,k_3,k_4) \in \Z^4: k_1^2+k_2^2+k_3^2+k_4^2=N,k_i \equiv x_i \mod m\}.\]
Since 
\[r_4(N)=\sum_{(x_1,x_2,x_3,x_4)\in E} r_4(N,x_1,x_2,x_3,x_4),\]
there must be some $(x_1,x_2,x_3,x_4)\in E$ for which $r_4(N,x_1,x_2,x_3,x_4) \geq r_4(N)/m^4$. Recall Euler's four-square identity
\begin{align*}
(y_{1}^{2} + y_{2}^{2} &+ y_{3}^{2} + y_{4}^{2})(x_{1}^{2} + x_{2}^{2} + x_{3}^{2} + x_{4}^{2}) \\[5pt]
&= (x_{1}y_{1} + x_{2}y_{2} + x_{3}y_{3} + x_{4}y_{4})^{2} 
+ (-x_{1}y_{2} + x_{2}y_{1} - x_{3}y_{4} + x_{4}y_{3})^{2} \\[3pt]
&\qquad + (-x_{1}y_{3} + x_{2}y_{4} + x_{3}y_{1} - x_{4}y_{2})^{2} 
+ (-x_{1}y_{4} - x_{2}y_{3} + x_{3}y_{2} + x_{4}y_{1})^{2}.
\end{align*}
Suppose that we can find integers $y_j \in [0,m)$ such that
\begin{equation}\label{simeq1}
y_{1}^{2} + y_{2}^{2} + y_{3}^{2} + y_{4}^{2} \equiv 1 {\hskip -6pt}\pmod{m}
\end{equation}
and
\begin{equation}\label{simeq2}
x_{1} y_{1} + x_{2} y_{2} + x_{3} y_{3} + x_{4} y_{4} \equiv 0 {\hskip -6pt}\pmod{m},
\end{equation}
given fixed integers $x_{j}$ satisfying
\begin{equation}\label{startupeq0}
x_{1}^{2} + x_{2}^{2} + x_{3}^{2} + x_{4}^{2} \equiv a {\hskip -6pt}\pmod{m}.
\end{equation}
Let $S:\Z^4 \to \Z^4$ be given by 
\[S(x_1,x_2,x_3,x_4)=\begin{pmatrix}
y_1 & y_2 & y_3 & y_4\\
-y_2 & y_1 & -y_4 & y_3\\
-y_3 & y_4 & y_1 & -y_2\\
-y_4 & -y_3 & y_2 & y_1
\end{pmatrix}
\begin{pmatrix}
x_1 \\
x_2 \\
x_3\\
x_4
\end{pmatrix}.\]

The determinant of the above matrix is $(y_1^2+y_2^2+y_3^2+y_4^2)^2$. Recalling \eqref{simeq1}, we know that this is non-zero. It follows that $S$ is injective, and so we have $\geq r_4(N)/m^4$ representations of $N(y_1^2+y_2^2+y_3^2+y_4^2)$ as a sum of four squares, the first of which is $0 \mod m$. 

It will follow that 
\begin{equation}\label{Chowlatrick}
\sum_{\substack{|k| \leq \sqrt{N(y_1^2+y_2^2+y_3^2+y_4^2)} \\ k\equiv 0 \mod m}}r_3(N(y_1^2+y_2^2+y_3^2+y_4^2)-k^2) \geq r_4(N)/m^4,
\end{equation}
and that there must be some integer $n \leq N(y_1^2+y_2^2+y_3^2+y_4^2)$ for which we have 
\[
r_3(n) \geq \dfrac{r_4(N)}{(2\sqrt{N(y_1^2+y_2^2+y_3^2+y_4^2)}+1)\,m^3}
\gg_{m} \sqrt{N}\,\log\log N
\] 
and such an integer $n$ is congruent to $a \mod m$ by construction. 
Thus, Theorem \ref{mainresult} will follow provided we we can  choose $y_1,y_2,y_3,y_4 \in \{0,1,\ldots,m-1\}$ that satisfy equations \eqref{simeq1} and \eqref{simeq2}

\section{Contruction of the solution}\label{construction}

Our goal in this section is to show that we can find $y_1,y_2,y_3,y_4$ that satisfy equations \eqref{simeq1} and \eqref{simeq2}. We start by showing that we can make a few simplifications. The first simplification we make is that we may suppose that $a$ is square-free.  A somewhat obvious approach to this is to write $a=d_1d_2^2$, with $d_1$ square-free, and then to replace finding representations of integers in the class $a \mod{m}$ as the sum of $3$ squares with finding representations of integers in the class $d_{1} \mod{m}$ as the sum of $3$ squares.  For each such representation $x_{1}^{2} + x_{2}^{2} + x_{3}^{2}$ of integers in the class $d_{1} \mod{m}$, we can then consider $(d_{2}x_{1})^{2} + (d_{2}x_{2})^{2} + (d_{2}x_{3})^{2}$ in the class $a \mod{m}$.  The difficulty with this idea is illustrated by the example $a = 28$ and $m = 40$, so $d_{1} = 7$ and $d_{2} = 2$.  There are no integers in the class $7 \mod{40}$ which can be written as a sum of $3$ squares, so we will not be able to show there are many representations of integers in the class $28 \mod{40}$ as a sum of $3$ squares by first demonstrating such representations exist in the class $7 \mod{40}$.  And note that $68 = 6^{2} + 4^{2} + 4^{2}$, so the class $28 \mod{40}$ should have integers with many representations as a sum of $3$ squares.  

\begin{lem}\label{squarefreelemma}
Let $m$ be a positive integer, and let $0 \leq a < m$ be an integer.  Suppose that $a$ is a sum of $3$ squares modulo $m$.   
Then we can find a square-free integer $a' \in [1,5m]$ such that $a'$ is a sum of $3$ squares modulo $m$ and if there are infinitely many $n' \equiv a' \mod{m}$ such that $r_3(n') \gg_m \sqrt{n'}\log\log n'$, then there are infinitely many integers $n$ such that $n \equiv a \mod m$ and $r_3(n) \gg_m \sqrt{n}\log\log n$. 
\end{lem}

\begin{proof}
We may and do suppose instead that $a \in [1,m]$ (in other words, we replace $a$ with $m$ in the case that $a = 0$).  
Let $r$ and $s$ be the nonnegative integers for which $2^{r} \,\Vert\, a$ and $2^{s} \,\Vert\, m$, and write $a = 2^{r} a_{0}$ and $m = 2^{s} m_{0}$ with $a_{0}$ and $m_{0}$ odd integers.  
If $r$ is odd, then we can write $a=d_1d_2^2$, with $d_1$ square-free and $2 \,\Vert\, d_{1}$.  
In this case, we see, by Legendre's Three-Square Theorem, that $d_{1}$ can be written as a sum of $3$ squares.  We take $a' = d_{1}$ and note $a' \in [1,m]$ and that $a$ can be written as $a'$ times a square modulo $m$.   

Now, suppose $r$ is even.
Consider the case that $s \le r$.  Then $a + 2^{r-s+1} m = 2^{r} (a_{0} + 2 m_{0})$ and $a + 2^{r-s+2} m = 2^{r} (a_{0} + 4 m_{0})$ are such that the odd factors $a_{0} + 2 m_{0}$ and $a_{0} + 4 m_{0}$ are distinct modulo $4$.  As a consequence, $a_{0} + 2 \ell m_{0}$ is $1$ modulo $4$ for some $\ell \in \{ 1,2 \}$.  We write $a_{0} + 2 \ell m_{0} = d'_{1} (d'_{2})^{2}$ with $d'_1$ odd and square-free and $d'_{2}$ odd.  Then $(d'_{2})^{2} \equiv 1 \pmod{4}$ so $d'_{1} \equiv 1 \pmod{4}$.  Then Legendre's Three-Square Theorem implies that $d'_{1}$ can be written as a sum of $3$ squares.  Observe that 
\[
d'_{1} \le a_{0} + 2 \ell  m_{0} \le 5m 
\qquad \text{and} \qquad
a \equiv 2^{r} (a_{0} + 2 \ell  m_{0}) \equiv d'_{1} (d'_{2} 2^{r/2})^{2} \mod{m}.  
\]
Taking $a' = d'_{1}$, then we see that $a' \in [1,5m]$ and again $a$ can be written as $a'$ times a square modulo $m$.   

Next, suppose $r$ is even and $s \in \{r+1, r+2 \}$.  Then instead we consider $a + m = 2^{r} (a_{0} + 2^{s-r} m_{0})$ and $a + 2m = 2^{r} (a_{0} + 2^{s-r+1} m_{0})$.  Here, the odd factors $a_{0} + 2^{s-r} m_{0}$ and 
$a_{0} + 2^{s-r+1} m_{0}$ are distinct modulo $8$ so one of these two odd numbers is not $7$ modulo $8$.  Taking this number to be $a_{0} + 2 \ell m_{0}$, where $\ell \in \{ 1, 2, 4 \}$, then we write it as $d'_{1} (d'_{2})^{2}$ and take $a' = d'_{1}$ as before.  In this case, similar to the above, we have $a'$ can be written as a sum of $3$ squares, 
\[
d'_{1} \le a_{0} + 2 \ell  m_{0} \le m + \ell m \le 5m
\qquad \text{and} \qquad
a \equiv 2^{r} (a_{0} + 2 \ell  m_{0}) \equiv d'_{1} (d'_{2} 2^{r/2})^{2} \mod{m}.  
\]
Thus, again, $a' \in [1,5m]$ and $a$ can be written as $a'$ times a square modulo $m$.   

Next, suppose $r$ is even and $s \ge r+3$.
We claim in this case, we can arrive at a similar value of $a'$ by writing $a=d_1d_2^2$ with $d_{1}$ square-free and odd and taking $a' = d_{1} \in [1,m]$.  In other words, we claim that this choice of $a'$ is a sum of $3$ squares modulo $m$ and $a$ is $a'$ times a square modulo $m$.  That $a$ is $a'$ times a square modulo $m$ follows from $a=d_1d_2^2$.  Assume $a' = d_{1}$ is not the sum of $3$ squares modulo $m$.  Then $d_{1}$ cannot be a sum of $3$ squares over $\mathbb Z$.  As $d_{1}$ is odd, this implies by Legendre's Three-Square Theorem that $d_{1} \equiv 7 \pmod{8}$.  Note that $2^{r+3}$ divides $m$.  By the conditions of Lemma~\ref{squarefreelemma}, there are integers $x_{1}, x_{2}$ and $x_{3}$ such that
\[
d_1d_2^2 = a \equiv x_{1}^{2} + x_{2}^{2} + x_{3}^{2} \pmod{2^{r+3}}.
\]
Since $r$ is even, we have $d_{2} = 2^{r/2} d'_{2}$, where $d'_{2}$ is odd and therefore has an inverse modulo $2^{r+3}$.  We deduce that
\begin{equation}\label{eqimpossible}
2^{r} d_{1} \equiv (x'_{1})^{2} + (x'_{2})^{2} + (x'_{3})^{2} \pmod{2^{r+3}},
\end{equation}
for some integers $x'_{1}, x'_{2}$ and $x'_{3}$.  This is impossible since $d_{1} \equiv 7 \pmod{8}$.  More precisely, if $r = 0$, then a congruence argument modulo $8$ shows \eqref{eqimpossible} does not hold; and if $r > 0$, then $r \ge 2$ and an argument modulo $4$ shows each $x'_{j}$ is even and we can divide through by $4$ (adjusting the squares on the right of \eqref{eqimpossible} and replacing $r$ with $r-2$ on the left of \eqref{eqimpossible} and in the modulus) and apply Fermat's method of descent.  Thus, we have a contradiction, and $a'$ is a sum of $3$ squares modulo $m$, and furthermore $a$ is $a'$ times a square modulo $m$.

Now that we have constructed $a'$ with $a \equiv a' d^{2} \mod{m}$ for some integer $d$ which we take to be in $[1,m]$, suppose we have $n \in \mathbb Z$ large, compared to $m$, for which $n \equiv a' \mod m$ and $r_3(n) \gg_{m} \sqrt{n}\log\log n$.  Notice that if $n=x_1^2+x_2^2+x_3^2$, then
\[
n d^{2} = (d x_1)^2+(d x_2)^2+(d x_3)^2.
\]
Hence, we see that
\[
r_3(n d^{2}) \geq r_3(n) \gg_m \sqrt{n}\,\log\log n 
\gg_m \sqrt{n m^2} \,\log\log (n m^{2})
\gg_m \sqrt{n d^2} \,\log\log (n d^{2}).
\]
Since $n d^{2} \equiv a' d^{2} \equiv a \mod m$, the lemma follows. 
\end{proof}

We now suppose, which we can do by Lemma~\ref{squarefreelemma}, that $a$ is square-free and $a \in [1,5m]$.  Next, we show that we may work with a certain multiple of $m$ instead of working with $m$.  The idea is that if $n \equiv a \mod {tm}$ for some positive integer $t$, then $n \equiv a \mod m$.  Thus, as long as $t$ is bounded as a function of $m$, then it suffices to show that $r_{3}(n) \gg_{m} \sqrt{n} \log\log n$ for infinitely many $n \equiv a \mod {tm}$.  But one has to be selective in the choice of $t$.  If $a = 7$ and $n = 20$, for example, then there are $n \equiv a \mod {m}$ which can be written as a sum of $3$ squares, like $27 = 3^{2} + 3^{2} + 3^{2}$, but every $n \equiv a \mod {2m}$ is not a sum of $3$ squares since such $n$ are $7$ modulo $8$.  So we cannot hope in this case to prove $r_{3}(n) \gg_{m} \sqrt{n} \log\log n$ for infinitely many $n \equiv a \mod {m}$ by restricting to $n \equiv a \mod {2m}$.  With this in mind, we define $t$ by
\[
tm = 7 \,\prod_{p \mid a} p^{\max\{2,\nu_{p}(m)\}} \cdot \prod_{\substack{p \mid m \\ p \nmid a}} p^{\nu_{p}(m)},
\]
where $\nu_{p}(m)$ denotes the exponent in the largest power of $p$ dividing $m$.  The purpose of the factor $7$ is simply to ensure $t \ge 7$ so that $a \le 5m < tm$.  Recall that $a$ is square-free and, hence, is $1$, $2$ or $3$ modulo $4$. 
If $\nu_{2}(tm) \le 2$, then $a$ and $a + 2^{2-\nu_{2}(tm)} tm$ are distinct modulo $8$ but the same and non-zero modulo $4$.  Legendre's Three-Square Theorem implies, in this case, that there are $n \equiv a \mod{tm}$ that are sums of $3$ squares.
If $\nu_{2}(tm) \ge 3$, then the definition of $t$ above implies $m$ is divisible by $8$.
Since $a$ is square-free, if $a \not\equiv 7 \mod{8}$, then again Legendre's Three-Square Theorem implies that there are $n \equiv a \mod{tm}$ that are sums of $3$ squares.  If $a \equiv 7 \mod{8}$, then $a$ is not a sum of $3$ squares and furthermore $a$ is not a sum of $3$ squares modulo $m$.  In this case, Theorem~\ref{mainresult} does not apply.  Thus, we can replace $m$ by $tm$, and the condition in Theorem~
\ref{mainresult} that $a$ can be written as a sum of $3$ squares modulo $m$ will still apply with $m$ replaced by $tm$.  As a consequence of this replacement and the definition of $t$, we may and do suppose now that the square of any prime divisor of the square-free integer $a$ divides $m$ and also that $a \in [1,m)$.  

To clarify in advance the purpose of the above set-up, we have the congruence \eqref{startupeq0}.  
With the above set-up, we can deduce that, for each prime divisor $p$ of $a$, some $x_{j}$ satisfies $p \nmid x_{j}$, as it is impossible for $p^{2}$ to divide both $m$ and the left-hand side of the congruence \eqref{startupeq0} but not $a$.
Furthermore, if $p \mid m$ and $p \nmid a$, then \eqref{startupeq0} implies $p \nmid x_{j}$ for some $j$.  
Thus, in general, for each prime $p$ dividing $m$, there must be some $x_{j}$ not divisible by $p$.

Given Section~\ref{overview}, our goal will be to apply the Chinese Remainder Theorem after showing that for a fixed $p^{e}$, with $p$ a prime and $e \ge 1$, we can find integers $y_{j}$ such that
\begin{equation}\label{simeq3}
y_{1}^{2} + y_{2}^{2} + y_{3}^{2} + y_{4}^{2} \equiv 1 {\hskip -6pt}\pmod{p^{e}}
\end{equation}
and
\begin{equation}\label{simeq4}
x_{1} y_{1} + x_{2} y_{2} + x_{3} y_{3} + x_{4} y_{4} \equiv 0 {\hskip -6pt}\pmod{p^{e}},
\end{equation}
given integers $x_{j}$ satisfying
\begin{equation}\label{startupeq1}
x_{1}^{2} + x_{2}^{2} + x_{3}^{2} + x_{4}^{2} \equiv a {\hskip -6pt}\pmod{p^{e}},
\end{equation}
where $a$ is square-free, some $x_{j}$ is not divisible by $p$, and $e \ge 2$ if $p \mid a$.

We begin by addressing the case that $p = 2$.
Here, $4 \nmid a$ and some $x_{j}$ is odd.
If $a$ is odd, then \eqref{startupeq1} implies that some $x_{j}$ is even.
If $a$ is even, then $e \ge 2$.  
Since $4 \nmid a$ and \eqref{startupeq1} holds, we see again that some $x_{j}$ is even.
Recall that each $x_{j}$ is in the interval $[1,m]$ and, hence, non-zero.
For $j \in \{ 1, 2, 3, 4 \}$, define $r_{j} \in \mathbb Z^{+} \cup \{ 0 \}$ by $2^{r_{j}} \,\Vert\, x_{j}$, 
and write $x_{j} = 2^{r_{j}} x'_{j}$ where $x'_{j} \in \mathbb Z$.
Using the symmetry in \eqref{simeq3}, \eqref{simeq4} and \eqref{startupeq1}, we take 
\[
0 = r_{1} \le r_{2} \le r_{3} \le r_{4}
\qquad \text{and} \qquad
r_{4} \ge 1.
\]
Let $\ell = \lfloor e/2 \rfloor$ so that $e \in \{ 2\ell, 2\ell+1 \}$.
Our goal for $p = 2$ is to show there exist integers $y_{j}$ satisfying \eqref{simeq3} and \eqref{simeq4} and such that $2^{r_{4}} \,\Vert\, y_{1}$, $y_{2} = 0$ and $y_{4}$ is odd.

For $1 \le e \le r_{4}+1$, we take $y_{1} = 2^{r_{4}} \ge 2$, $y_{2} = y_{3} = 0$, and $y_{4} = 1$.  
One checks directly that \eqref{simeq3} and \eqref{simeq4} hold with $p = 2$.  Here, we use that two integers which are divisible by $2^{r_{4}}$ and not by $2^{r_{4}+1}$ have a sum divisible by $2^{r_{4}+1}$.  

We finish the proof by induction on $e$.  For $p = 2$ and for some $e \ge r_{4}+1$, suppose we know that there exist integers $y_{j}$ satisfying \eqref{simeq3} and \eqref{simeq4} and such that $2^{r_{4}} \,\Vert\, y_{1}$, $y_{2} = 0$ and $y_{4}$ is odd.  Fix such $y_{j}$. 
Observe that this implies 
\begin{equation}\label{simeq5}
y_{1}^{2} + y_{2}^{2} + y_{3}^{2} + y_{4}^{2} \equiv u {\hskip -6pt}\pmod{2^{e+1}},
\quad \text{where } u \in \{ 1, 2^{e}+1 \},
\end{equation}
and
\begin{equation}\label{simeq6}
x_{1} y_{1} + x_{2} y_{2} + x_{3} y_{3} + x_{4} y_{4} \equiv v {\hskip -6pt}\pmod{2^{e+1}},
\quad \text{where } v \in \{ 0, 2^{e} \}.
\end{equation}
We make the following further observations.
\begin{itemize}
\item
If $u = 1$ and $v = 0$, then we are done; there exist $y_{j}$ satisfying \eqref{simeq3} and \eqref{simeq4} with $p^{e}$ replaced by $2^{e+1}$ as we want.  
\item
If $u = 1$ and $v = 2^{e}$, then we change $y_{1}$ to $y_{1} + 2^{e}$ to obtain the necessary $y_{j}$ satisfying \eqref{simeq3} and \eqref{simeq4} with $p^{e}$ replaced by $2^{e+1}$.
\item
If $u = 2^{e}+1$, $v = 0$ and $r_{4} \ge 2$ or if $u = 2^{e}+1$, $v = 2^{e}$ and $r_{4} = 1$, then we change instead $y_{4}$ to $y_{4} + 2^{e-1}$ to obtain the necessary $y_{j}$ satisfying \eqref{simeq3} and \eqref{simeq4} with $p^{e}$ replaced by $2^{e+1}$.
\item
If $u = 2^{e}+1$, $v = 0$ and $r_{4} = 1$ or if $u = 2^{e}+1$, $v = 2^{e}$ and $r_{4} \ge 2$, then we change instead $y_{1}$ to $y_{1} + 2^{e}$ and $y_{4}$ to $y_{4} + 2^{e-1}$ to obtain the necessary $y_{j}$ satisfying \eqref{simeq3} and \eqref{simeq4} with $p^{e}$ replaced by $2^{e+1}$.
\end{itemize}
\noindent
Thus, in any case, we see that there exist integers $y_{j}$ satisfying \eqref{simeq3} and \eqref{simeq4} and such that $2^{r_{4}} \,\Vert\, y_{1}$, $y_{2} = 0$ and $y_{4}$ is odd with $p^{e}$ replaced by $2^{e+1}$, completing the induction and establishing that there exist integers $y_{j}$ satisfying \eqref{simeq3} and \eqref{simeq4} with $p^{e}$ replaced by any power of $2$.  

For odd primes, we will make use of the following lemma.

\begin{lem}\label{newlemma3.1shortened}
Let $p$ be an odd prime, and let $e \in \mathbb Z^{+}$.
Let $a'$, $b'$ and $d$ be integers which are coprime to $p$.  
Let $f(x_{1},x_{2}) = a' x_{1}^{2} + b' x_{2}^{2}$.  
Then $f(x_{1},x_{2}) \equiv d  {\hskip -2pt}\pmod{p^{e}}$ has a solution.
\end{lem}

\begin{proof}
We start with $e = 1$ and apply the Cauchy-Davenport inequality \cite[Theorem 5.4]{TaoVu}.  In particular, if $Q$ is the set of $(p+1)/2$ squares modulo $p$ and $N$ is the union of the set of non-squares modulo $p$ and $\{ 0 \}$, which is also of size $(p+1)/2$, then the values of $f(x_{1},x_{2}) = a' x_{1}^{2} + b' x_{2}^{2}$ modulo $p$ are the same as the elements in one of the sets $Q + Q$, $Q + N$ or $N + N$ modulo $p$, depending on whether $a'$ and $b'$ are quadratic residues or not.  Thus, the Cauchy-Davenport inequality implies that every residue class modulo $p$ can be represented as a value of $f$.  In particular, there are integers $x_{0}$, $y_{0}$, $x'_{0}$ and $y'_{0}$ such that $f(x_{0},y_{0})$ is a non-zero square modulo $p$ and $f(x'_{0},y'_{0})$ is a non-square modulo $p$ (necessarily, also non-zero modulo $p$). 
Fix such $x_{0}$, $y_{0}$, $x'_{0}$ and $y'_{0}$.

Now, consider more generally $e \ge 1$.  Note that a non-zero integer modulo $p$ is a square modulo $p$ if and only if it is a square modulo $p^{e}$.  Let $g$ be a primitive root modulo $p^{e}$.  We deduce that there are nonnegative integers $k$ and $\ell$ such that 
\[
f(x_{0},y_{0}) \equiv g^{2k} {\hskip -6pt}\pmod{p^{e}} 
\qquad \text{and} \qquad
f(x'_{0},y'_{0}) \equiv g^{2\ell+1} {\hskip -6pt}\pmod{p^{e}}.
\]  
Since $f(gx,gy) = g^{2} f(x,y)$, by multiplying both sides of these congruences by $g^{2}$ a number of times, we see that each square and non-square modulo $p^{e}$, not divisible by $p$, is a value of $f$ modulo $p^{e}$.  Since $d$ is relatively prime to $p$, we deduce the number $d$ is a value of $f$ modulo $p^{e}$.
\end{proof}

Recall that some $x_{j}$ is not divisible by $p$. 
By the symmetry of the congruences \eqref{simeq3}, \eqref{simeq4} and \eqref{startupeq1}, we may and do suppose that $p \nmid x_{1}$.  Fix $j \in \{ 2, 3, 4 \}$, and set $y_{j} = 0$.  As $x_{1}$ is invertible modulo $p^{e}$, we can solve for $y_{1}$ in \eqref{simeq4}.  Substituting this choice of $y_{j}$ and $y_{1}$ into \eqref{simeq3} and simplifying, we obtain for $j \in \{ 2, 3, 4 \}$, respectively, that  
\begin{gather}
(x_{1}^{2} + x_{3}^{2}) y_{3}^{2} + 2 x_{3} x_{4} y_{3} y_{4} + (x_{1}^{2} + x_{4}^{2}) y_{4}^{2} \equiv x_{1}^{2} {\hskip -6pt}\pmod{p^{e}} \label{y2zero} \\[5pt]
(x_{1}^{2} + x_{2}^{2}) y_{2}^{2} + 2 x_{2} x_{4} y_{2} y_{4} + (x_{1}^{2} + x_{4}^{2}) y_{4}^{2} \equiv x_{1}^{2} {\hskip -6pt}\pmod{p^{e}} \label{y3zero} \\[5pt]
(x_{1}^{2} + x_{2}^{2}) y_{2}^{2} + 2 x_{2} x_{3} y_{2} y_{3} + (x_{1}^{2} + x_{3}^{2}) y_{3}^{2} \equiv x_{1}^{2} {\hskip -6pt}\pmod{p^{e}}. \label{y4zero}
\end{gather}
We will be content if we can solve for the $y_{i}$'s in any one of these congruences because then we can back-track and determine the remaining $y_{i}$, namely $y_{1}$ will come from \eqref{simeq4} and the last remaining $y_{i}$ will be $0$.

To apply Lemma~\ref{newlemma3.1shortened}, we simply want to complete a square.  
First, suppose that $p \nmid (x_{1}^{2}+x_{3}^{2})$.  Then we can rewrite \eqref{y2zero} as
\begin{equation}\label{y2zeroresult}
(x_{1}^{2} + x_{3}^{2}) \bigg( y_{3} + \dfrac{x_{3} x_{4}}{x_{1}^{2} + x_{3}^{2}} y_{4} \bigg)^2 
+ \bigg( - \dfrac{x_{3}^{2} x_{4}^{2}}{x_{1}^{2} + x_{3}^{2}} + x_{1}^{2} + x_{4}^{2} \bigg) y_{4}^{2} \equiv x_{1}^{2} {\hskip -4pt}\pmod{p^{e}}.
\end{equation}
By \eqref{startupeq1}, we have
\[
- \dfrac{x_{3}^{2} x_{4}^{2}}{x_{1}^{2} + x_{3}^{2}} + x_{1}^{2} + x_{4}^{2} 
= \dfrac{x_{1}^{2} (x_{1}^{2} + x_{3}^{2} + x_{4}^{2})}{x_{1}^{2} + x_{3}^{2}}
\equiv \dfrac{x_{1}^{2} (a - x_{2}^{2})}{x_{1}^{2} + x_{3}^{2}}  {\hskip -4pt}\pmod{p^{e}}.
\]
Thus, if also $p \nmid (a - x_{2}^{2})$, then Lemma~\ref{newlemma3.1shortened} implies that we can solve \eqref{y2zeroresult} for $y_{4}$ modulo $p^{e}$ and then $y_{3}$, followed by $y_{2}$ and $y_{1}$.
Since $p \nmid (x_{1}^{2}+x_{3}^{2})$, we also could rewrite \eqref{y4zero} as
\begin{align*}
(x_{1}^{2} + x_{3}^{2}) \bigg( y_{3} + \dfrac{x_{2} x_{3}}{x_{1}^{2} + x_{3}^{2}} y_{2} \bigg)^2 
+ \bigg( - \dfrac{x_{2}^{2} x_{3}^{2}}{x_{1}^{2} + x_{3}^{2}} + x_{1}^{2} + x_{2}^{2} \bigg) y_{2}^{2} \equiv x_{1}^{2} {\hskip -4pt}\pmod{p^{e}},
\end{align*}
where
\[
- \dfrac{x_{2}^{2} x_{3}^{2}}{x_{1}^{2} + x_{3}^{2}} + x_{1}^{2} + x_{2}^{2} 
= \dfrac{x_{1}^{2} (x_{1}^{2} + x_{2}^{2} + x_{3}^{2})}{x_{1}^{2} + x_{3}^{2}}
\equiv \dfrac{x_{1}^{2} (a - x_{4}^{2})}{x_{1}^{2} + x_{3}^{2}}  {\hskip -4pt}\pmod{p^{e}},
\]
and we can solve for the $y_{j}$ provided $p \nmid (a - x_{4}^{2})$.
Thus, we are done in the case that $p \nmid (x_{1}^{2}+x_{3}^{2})$ unless
\begin{equation}\label{x2x4congs}
x_{2}^{2} \equiv x_{4}^{2} \equiv a {\hskip -6pt}\pmod{p}.
\end{equation}
So suppose \eqref{x2x4congs} holds.  
From \eqref{startupeq1}, we deduce
\[
x_{1}^{2} + x_{3}^{2} \equiv -a {\hskip -6pt}\pmod{p}.
\]
Since $p \nmid (x_{1}^{2}+x_{3}^{2})$, we deduce that $p \nmid a$.
From \eqref{x2x4congs}, we see now that $p \nmid x_{2}$ and $p \nmid x_{4}$.  
As \eqref{x2x4congs} implies $x_{2}^{2} + x_{4}^{2} \equiv 2a {\hskip -2pt}\pmod{p}$ and $p$ is odd, we also see that $p \nmid (x_{2}^{2} + x_{4}^{2})$.

Still under the condition $p \nmid (x_{1}^{2}+x_{3}^{2})$, we approach solving for the $y_{j}$ in \eqref{simeq3} and \eqref{simeq4} a different way.  Since $p \nmid (x_{1} x_{4})$, for every choice of integers $r$ and $s$, we can set
\[
y_{1} \equiv \dfrac{-rx_{3}}{x_{1}} {\hskip -7pt}\pmod{p^{e}}, \quad
y_{2} \equiv s {\hskip -7pt}\pmod{p^{e}}, \quad
y_{3} \equiv r {\hskip -7pt}\pmod{p^{e}}, \quad
y_{4} \equiv \dfrac{-s x_{2}}{x_{4}} {\hskip -7pt}\pmod{p^{e}}
\]
to ensure that \eqref{simeq4} holds.  
Substituting these into \eqref{simeq3} and simplifying, we obtain
\[
x_{4}^{2} (x_{1}^{2}+x_{3}^{2}) r^{2} + x_{1}^{2} (x_{2}^{2}+x_{4}^{2}) s^{2} \equiv x_{1}^{2}x_{4}^{2} {\hskip -6pt}\pmod{p^{e}}.
\]
Since $p$ does not divide the coefficients of $r^{2}$ and $s^{2}$ as well as $x_{1}^{2}x_{4}^{2}$ in this congruence, Lemma~\ref{newlemma3.1shortened} implies that we can solve for $r$ and $s$ modulo $p^{e}$.  These choices for $r$ and $s$ then provide us values for the $y_{j}$ such that \eqref{simeq3} and \eqref{simeq4} hold.  

The above holds provided $p \nmid (x_{1}^{2}+x_{3}^{2})$.  Thus, using $p \nmid x_{1}$, we are able to show that we can find $y_{j}$ satisfying \eqref{simeq3} and \eqref{simeq4} unless $p \mid (x_{1}^{2}+x_{3}^{2})$.  By the symmetry of \eqref{simeq3}, \eqref{simeq4} and \eqref{startupeq1}, we deduce that we can find $y_{j}$ satisfying \eqref{simeq3} and \eqref{simeq4} unless $p \mid (x_{1}^{2}+x_{2}^{2})$ and $p \mid (x_{1}^{2}+x_{4}^{2})$.  Furthermore, since $p \nmid x_{1}$, from $p \mid (x_{1}^{2}+x_{3}^{2})$, we deduce $p \nmid x_{3}$.  Similarly, $p \nmid x_{2}$ and $p \nmid x_{4}$.  We can now use the symmetry of \eqref{simeq3}, \eqref{simeq4} and \eqref{startupeq1} again to deduce $p \mid (x_{i}^{2}+x_{j}^{2})$ for all choices of $i$ and $j$ in $\{ 1,2,3,4 \}$ with $i \ne j$.  However, this is impossible as then
\[
(x_{1}^{2}+x_{2}^{2}) + (x_{1}^{2}+x_{3}^{2}) - (x_{2}^{2}+x_{3}^{2}) = 2 x_{1}^{2}
\]
is divisible by $p$, since the left-hand side is, but $p$ does not divide $2 x_{1}^{2}$ since $p$ is odd and $p \nmid x_{1}$.  We deduce in fact that one of $p \nmid (x_{1}^{2}+x_{2}^{2})$, $p \nmid (x_{1}^{2}+x_{3}^{2})$ and $p \nmid (x_{2}^{2}+x_{3}^{2})$ must hold, allowing us to find $y_{j}$ satisfying \eqref{simeq3} and \eqref{simeq4}.  

\section{Proof of Corollary \ref{maincor}}\label{corollaryproofs}
Suppose first that $\nu_{2}(m) \le 1$.  Let $a \in \{ 0, 1, \ldots, m-1 \}$. 
We have $\nu_{2}(a+im) \le 1$ for some $i \in \{ 0,1 \}$.  
Then the numbers $a+im$ and $a+im+(4m/2^{\nu_{2}(m)})$ are the same modulo $4$, not divisible by $4$, and incongruent modulo $8$.
By Legendre’s Three-Square Theorem, at least one of these, say $b$, can be expressed as a sum of $3$ squares over $\mathbb Z$ and, hence, $b$ can be expressed as a sum of $3$ squares modulo $4m$.  
By Theorem~\ref{mainresult}, there are infinitely many $n$ such that $n \equiv b \mod {4m}$ and $r_3(n) \gg_m \sqrt{n} \log\log n$.  Note that $b$ is not divisible by $4$, so $n$ is not divisible by $4$.  Also, $n \equiv b \equiv a \mod{m}$ and $r_{3}(n) > 0$.  Thus, by relationship \eqref{Hurwitz_class}, we deduce again that $H(n) \gg_m \sqrt{n}\log\log n$ for infinitely many $n \equiv a \pmod{m}$.

Now, suppose that $m$ is divisible by $4$.  Let $a$ be an integer not divisible by $4$. Suppose that $a$ is representable as a sum of $3$ squares modulo $m$. Let $n$ be given by Theorem \ref{mainresult} so that $n \equiv a \mod m$ and $r_3(n) \gg_m \sqrt{n}\log\log n$. Because $4 \nmid a$, we know that $n \not\equiv 0 \mod 4$. Furthermore, we know that $r_3(n) >0$. By relationship \eqref{Hurwitz_class}, it follows that $H(n) \gg_{m} \sqrt{n} \log\log n$.

\section*{Acknowledgements}
The second author would like to thank Tapas Bhowmik for bringing this problem to his attention. The authors would like to thank Andrew Granville, Roger Heath-Brown, Frank Thorne, Ognian Trifonov, and Wei-Lun Tsai for helpful conversations. 

\section*{Funding}
Jonah Klein is funded by a scholarship from the Natural Sciences and Engineering Research Council of Canada (NSERC). 

\bibliographystyle{plain}

\begin{thebibliography}{99}

\bibitem{berndtmoree}
B.~C.~Berndt and P.~Moree, {\it Sums of two squares and the tau-function: Ramanujan's trail}, 	arXiv:2409.03428.

\bibitem{Brink}
D.~Brink, P.~Moree and R.~Osburn, 
{\it Principal forms $X^2 + nY^2$ representing many integers}, 
Abh.~Math.~Sem. Univ.~Hambg.~81 (2011), 129--139.

\bibitem{Chowla}
S.~Chowla,
{\it On the k-analogue of a result in the theory of the Riemann zeta function}, Mathematische Zeitschrift 38 (1934), 483--487.

\bibitem{GranSound}
A. Granville, K. Soundararajan, 
{\it The distribution of values of L(1, $\chi_d$ ).}
Geom. Funct. Anal. 13 (2003): 992–1028 

\bibitem{grosswald}
E.~Grosswald, 
Representations of integers as sums of squares,
Springer-Verlag, New York, 1985.

\bibitem{Joshi}
P.~Joshi, 
{\it The size of $L(1, \chi)$ for real nonprincipal residue characters $\chi$ with prime modulus}, J. Number
Theory 2 (1970), 58--73.

\bibitem{Legendre}
A.-M.~Legendre, Essai sur la théorie des nombres, Paris, An VI (1797–1798),  pp. 398--399.

\bibitem{Mortenson}
E.~T.~Mortenson, 
{\it A {K}ronecker-type identity and the representations of a number as a sum of three squares},
Bull. Lond. Math. Soc.~49 (2017), 770--783.

\bibitem{TaoVu}
T.~Tao and V.~Vu, 
Additive Combinatorics, 
Cambridge University Press, Cambridge, 2006.


\end{thebibliography}

\end{document}